\documentclass[12pt]{amsart}%
\usepackage[utf8]{inputenc}
\usepackage{amsmath,amssymb,mathrsfs, amsfonts, amsthm}
\usepackage{amsthm}
\usepackage{comment}
\usepackage{subfig}
\usepackage{graphicx}
\usepackage{todonotes}
\usepackage[margin=1in]{geometry}
\usepackage{color}
\usepackage{amssymb}
\usepackage{amsmath}
\usepackage{amsthm}
\usepackage{url}
\usepackage{framed}
\usepackage{enumitem}
\usepackage{multirow}
\usepackage{setspace}
\usepackage{ulem}
\usepackage{amsfonts}%
\providecommand{\U}[1]{\protect\rule{.1in}{.1in}}
\allowdisplaybreaks[3]
\newtheorem{theorem}{Theorem}

\newtheorem{proposition}[theorem]{Proposition}

\newtheorem{lemma}{Lemma}

\newtheorem{conjecture}{Conjecture}

\begin{document}
\title{The HRT conjecture for a special configuration}
\date{\today}
\author{Kasso A.~Okoudjou}
\address{Department of Mathematics, Tufts University, Medford MA 02131, USA}
\email{kasso.okoudjou@tufts.edu}
\author{Vignon Oussa}
\address{Department of Mathematics, Bridgewater State University, Bridgewater, MA 0235, USA}
\email{vignon.oussa@bridgew.edu}

\begin{abstract}
The following work explores a subcase of the Heil-Ramanathan-Topiwala (HRT) conjecture, which proposes that a set of any finite time-frequency shifts of a nonzero square-integrable function is linearly independent. We identify and discuss certain sufficient conditions, focusing on the rational dimension of a particular vector space and the size of the zero set of the Zak transform under which the conjecture remains valid. A notable implication of our main result is the successful resolution of a case of the HRT Subconjecture, originally proposed by Chris Heil \cite{Heil2006}.

\end{abstract}
\maketitle

\section{Introduction}

For a nonzero function $f\in L^{2}\left(  \mathbb{R}\right)  ,$ a
time-frequency shift of $f$ is a function of the type
\[
M_{y}T_{x}f\left(  t\right)  =e^{-2\pi iyt}f\left(  t-x\right)  \text{
\ \ }\left(  x,y\right)  \in\mathbb{R}^{2}.
\]
The operator $M_{y}$ is called a modulation or frequency-shift operator, and
$T_{x}$ is an operator that acts by shifting $f$ in time domain. Generally,
modulation and translation operators do not commute, and in fact, these two
families of unitary operators generate a non-commutative group called the
Heisenberg group. The ubiquitous nature of the Heisenberg group and its
relevance across numerous subjects within harmonic analysis is a fact that has
been extensively documented in the literature \cite{bams/1183547543}. For
instance, the Heisenberg group is fundamental to the foundation of
time-frequency or Gabor analysis and is a source of important examples in
frame theory \cite{Groc2001, heil2010basis}. For a historical development, as
well as an in-depth treatment of Gabor analysis and connection with topics
such as frame theory, we refer the reader to \cite{FeiStr1, FeiStr2,
Groc2001,heil2007history}.

We recall that given a countable subset $\Lambda$ of $\mathbb{R}^{2},$ the
collection of functions
\[
\mathcal{G}\left(  f,\Lambda\right)  =\left\{  M_{y}T_{x}f:\left(  x,y\right)
\in\Lambda\right\}
\]
is called a Gabor system. For comparison, the system obtained by substituting
modulation for dilation operators is often called a wavelet system. The
underlying groups for these systems the Heisenberg group and the affine group
share many striking similarities, as documented in
\cite{OUSSA20191718,OUSSA20181202}. However, there are also several instances
in which these systems behave in drastically different ways. For example, it
is well-known that there exist finite systems produced by time-dilation
operators which are linearly dependent. This property guarantees the existence
of a scaling function as a solution to a suitable functional equation--a
central ingredient in the construction of multiresolution orthonormal wavelets
\cite{heil2010basis}. However, as of current date, the extent to which an
analogous statement holds for finite systems of time-frequency shifts remains
unknown. In fact, most progress made in this direction suggests that, any
finite system of time-frequency shifted copies of a nonzero function is
linearly independent. This conjecture, initially posited by Heil, Ramanathan,
and Topiwala, has come to be known as the HRT Conjecture \cite{HRT96}. It can
be stated as follows:

\begin{conjecture}
\cite{HRT96} (HRT Conjecture) Given a fixed finite set $\Lambda\subset
\mathbb{R}^{2} $, and a nonzero function $f \in L^{2}(\mathbb{R}) $, the
associated finite Gabor system, $\mathcal{G}(f, \Lambda) $ is linearly independent.
\end{conjecture}

For the specific case where all the points lie on the same line (which is
equivalent to stating that all modulation parameters are identical or all the
translation parameters are identical), a straightforward application of the
Fourier transform verifies the conjecture. Generally, however, the problem is
much more challenging than one might initially anticipate. As of the current
date, the most general result is attributed to Linnell \cite{MR1637388}, who
was able to corroborate the conjecture for the case where $\Lambda$ is a
shifted copy of a finite subset of a discrete subgroup of $\mathbb{R}^{2d}, d
\in\mathbb{N} $.

In terms of the number of points, the smallest case not addressed by Linnell's
result can only occur for specific configurations of a set $\Lambda$
containing four points. To date, the status of the HRT conjecture for sets of
four points can be summarized as follows.

\begin{proposition}
\label{summary4point} Let $\Lambda\subset\mathbb{R}^{2} $ be such that $\#
\Lambda= 4 $, and let $0 \neq f \in L^{2}(\mathbb{R}) $. The finite Gabor
system $\mathcal{G}(f, \Lambda) $ is linearly independent in each of the
following cases:

\begin{enumerate}
\item $\Lambda$ is a $(2,2)- $ configuration, that is, two of the points are
on a line and the other two on a parallel line, and $0 \neq f \in
L^{2}(\mathbb{R}) $ is arbitrary, \cite{Dem10, DemZah12};

\item $\Lambda$ is a $(1,3)- $ configuration, that is, three of the points are
on a line and the fourth point is off that line, and $0 \neq f \in\mathcal{S}
(\mathbb{R}) $ \cite{Dem10};

\item $\Lambda$ is an arbitrary set of four points, but extra restrictions are
imposed on $f $, e.g., \cite{BeBo13, BowSpee13, kao19}.
\end{enumerate}
\end{proposition}

Nevertheless, the general problem posed by the HRT conjecture for sets of four
points remains unresolved. The primary objective of this paper is to introduce
a new result in this direction.  To state the first result, we introduce the
Wiener amalgam space $W(\mathbb{R})$ and refer the reader to \cite[Section
6.1]{Groc2001} for further details.

A measurable function $f\in W(\mathbb{R})$ is said to belong to the Wiener
amalgam space if
\[
\left\Vert f\right\Vert _{W(\mathbb{R})}=\sum_{n\in\mathbb{Z}}\text{esssup}%
_{x\in\lbrack0,1]}\left\vert f(x+n)\right\vert <\infty.
\]
Denoting $C(\mathbb{R})$ as the space of continuous functions on $\mathbb{R}$,
we define $W_{0}(\mathbb{R})=C(\mathbb{R})\cap W(\mathbb{R})$ as the space of
continuous functions within the Wiener amalgam space. If $S(\mathbb{R})$
denotes the class of Schwartz functions, it can be proved that
\[
S(\mathbb{R})\subset W_{0}(\mathbb{R})\subset L^{1}(\mathbb{R})\cap
L^{2}(\mathbb{R}).
\]
The main aim of the following work is to establish the following:

\begin{theorem}
\label{thm:amalgampertub}
Consider a finite set $\{(x_{k}, y_{k}) : 1 \leq k \leq N\} \subset \mathbb{Z}^{2}$, and let $(\alpha, \beta)$ be a pair in $\mathbb{R}^{2}$ but not in $\mathbb{Z}^{2}$. Assume that a function $f$ either belongs to $W_{0}(\mathbb{R})$ or is a Schwartz function, and it satisfies the equation
\begin{equation}
\sum_{k=1}^{N}c_{k}M_{y_{k}}T_{x_{k}}f = M_{\alpha}T_{\beta}f
\end{equation}
where $c_{1}, \ldots, c_{N}$ are nonzero complex coefficients. Under these conditions, the following statements hold:

\begin{enumerate}
\item Let \(\alpha_{(-\beta,\alpha)}: [0,1)^2 \rightarrow [0,1)^2\) be the bijective map given by \(\alpha_{(-\beta,\alpha)}(z) = (z + (-\beta, \alpha)) \mod \mathbb{Z}^2\). Then the zero set of the Zak transform of \(f\) is necessarily \(\alpha_{(-\beta,\alpha)}\)-invariant.

    \item If the dimension of the vector space over the rationals $\mathbb{Q}$, spanned by $1$, $\alpha$, and $\beta$, is 2 (i.e., $\dim_{\mathbb{Q}}(\mathbb{Q} + \mathbb{Q}\alpha + \mathbb{Q}\beta) = 2$), then the zero set of the Zak transform of $f$ cannot be finite.
    
    \item If the dimension of this space is 3 (i.e., $\dim_{\mathbb{Q}}(\mathbb{Q} + \mathbb{Q}\alpha + \mathbb{Q}\beta) = 3$), then it necessarily follows that $f$ must be the zero function.
\end{enumerate}
\end{theorem}

One important aspect of Theorem \ref{thm:amalgampertub} is that it helps solve a particular case of the HRT Subconjecture, which Chris Heil described in \cite[Subconjecture 9.2]{Heil2006}. 

\section{Proof of Theorem~\ref{thm:amalgampertub}}

\label{sec2} To prove Theorem~\ref{thm:amalgampertub}, we proceed by
contradiction and assume that there exists a nonzero function $f\in
W_{0}(\mathbb{R})$ such that
\begin{equation}
\sum_{k=1}^{N}c_{k}M_{y_{k}}T_{x_{k}}f=M_{\alpha}T_{\beta}%
f\label{time_frequency}%
\end{equation}
for some nonzero complex coefficients $c_{1},\ldots,c_{N}$. Note that any
finite set of time-frequency operators parametrized by a collection of finite
points on an integer lattice in $\mathbb{R}^{2}$ forms a collection of
pairwise commuting operators. As such, the Spectral Theorem \cite{MR3112817}
guarantees the existence of a unitary operator that diagonalizes the operator
\begin{equation}
J=\sum_{k=1}^{N}c_{k}M_{y_{k}}T_{x_{k}}.\label{J}%
\end{equation}
In our case, this unitary operator is the Zak transform $Z:L^{2}%
(\mathbb{R})\rightarrow L^{2}([0,1]^{2})$, defined formally by
\begin{equation}
Zf(t,\omega)=\sum_{k\in\mathbb{Z}}f(t+k)e^{-2\pi k\omega i}%
\label{Zak_Transform}%
\end{equation}
The following section contains a summary of the main properties of the Zak
transform, and we refer the reader to \cite{Groc2001} for a complete
introduction to this transform.

\begin{lemma}
\label{prop:ZacT} For $\varphi\in L^{2}(\mathbb{R}) $ and $(t, \omega)
\in[0,1)^{2} $, we have:

\begin{itemize}
\item $[ZT\varphi](t, \omega) = e^{-2\pi\omega i} [Z\varphi](t, \omega) $.

\item $[ZM\varphi](t, \omega) = e^{-2\pi it} [Z\varphi](t, \omega) $.

\item For $\alpha, \beta> 0 $, $[ZM_{\alpha}T_{\beta}\varphi](t, \omega) =
e^{-2\pi i \alpha t} [Z\varphi](t - \beta, \omega+ \alpha) $.

\item For any integer $j$, $[ZT\varphi](t+j,\omega)=e^{2\pi\omega ji}%
[Z\varphi](t,\omega)$.

\item For any integer $j$, $[ZT\varphi](t,\omega+j)=[Z\varphi](t,\omega)$.

\item If $\varphi\in W_{0}(\mathbb{R}) $ or $\varphi\in S(\mathbb{R}) $, then
$Z\varphi$ is continuous on $\mathbb{R}^{2} $; see \cite[Lemma 8.2.1; Theorem
8.2.5]{Groc2001}.

\item If $\varphi$ is such that $Z\varphi$ is continuous on $\mathbb{R}^{2}$,
then $Z\varphi$ has a zero in $[0,1]^{2}$; see \cite[Lemma 8.4.2]{Groc2001}.
\end{itemize}
\end{lemma}

Applying the Zak transform to the equation $Jf=M_{\alpha}T_{\beta}f$ (see
(\ref{J}) and (\ref{Zak_Transform})) and letting $F=Zf$, for almost every
$(t,\omega)\in\lbrack0,1]^{2}$, we derive:
\begin{equation}
\sum_{k=1}^{N}c_{k}e^{-2\pi iy_{k}t}e^{-2\pi ix_{k}\omega}\cdot F(t,\omega
)=e^{-2\pi i\alpha t}\cdot F(t-\beta,\omega+\alpha),\label{line_3}%
\end{equation}
or equivalently,
\begin{equation}
p(t,\omega)\cdot F(t,\omega)=e^{-2\pi i\alpha t}\cdot F(t-\beta,\omega
+\alpha),\label{line_0}%
\end{equation}
where $p:\mathbb{R}^{2}\rightarrow\mathbb{C}$ is the non-zero trigonometric
polynomial given by
\begin{equation}
p(t,\omega)=\sum_{k=1}^{N}c_{k}e^{-2\pi iy_{k}t}e^{-2\pi ix_{k}\omega}%
=\sum_{k=1}^{N}c_{k}e^{-2\pi i\langle(t,\omega),(y_{k},x_{k})\rangle
}.\label{polynomial}%
\end{equation}
It follows that for almost every $(t,\omega)\in\lbrack0,1]^{2}$,
\[
\left\vert F((t,\omega)+\gamma)\right\vert =\left\vert p(t,\omega)\right\vert
\cdot\left\vert F(t,\omega)\right\vert
\]
where, for simplicity in notation, we set $\gamma=(-\beta,\alpha)$ and
$z=(t,\omega)$. Next, by iterating Equation \eqref{line_0}, we see that for
almost every $z\in\lbrack0,1]^{2}$ and for any positive integer $n\geq1$,
\begin{equation}
\left\vert F(z+n\gamma)\right\vert =\left\vert \prod_{j=0}^{n-1}%
p(z+j\gamma)\right\vert \cdot\left\vert F(z)\right\vert .\label{condition}%
\end{equation}
Using Lemma~\ref{prop:ZacT} and the assumption that $F=Zf$ is continuous on
$\mathbb{R}^{2}$, we conclude that there exists $\lambda\in\lbrack0,1]^{2}$
such that $F(\lambda)=Zf(\lambda)=0$. Then for each $n\geq1$, we have
\[
\left\vert F(\lambda+n\gamma)\right\vert =\left\vert \prod_{j=0}%
^{n-1}p(\lambda+j\gamma)\right\vert \cdot\left\vert F(\lambda)\right\vert =0.
\]
Let $\mathrm{zero}\left(  F\right)  $ be the zero set of the Zak transform of
$f.$ Next, let $\alpha_{\gamma}:\left[  0,1\right)  ^{2}\rightarrow\left[
0,1\right)  ^{2}$ be the bijective map given by $\alpha_{\gamma}\left(
z\right)  =\left(  z+\gamma\right)  \operatorname{mod}\mathbb{Z}^{2}.$ By
virtue of the fact that
\[
\left\vert F(\alpha_{\gamma}\left(  z\right)  )\right\vert =\left\vert
p(z)\right\vert \cdot\left\vert F(z)\right\vert
\]
it follows that $\mathrm{zero}\left(  F\right)  $ is $\alpha_{\gamma}%
$-invariant and consequently
\[
F\left(
{\displaystyle\bigcup\limits_{n\in\mathbb{N}}}
\alpha_{\gamma}^{n}\left(  \mathrm{zero}\left(  F\right)  \right)  \right)
=\left\{  0\right\}  .
\]
In light of this, the following is immediate:

\begin{itemize}
\item Suppose that $\dim_{\mathbb{Q}}(\mathbb{Q}+\mathbb{Q}\alpha
+\mathbb{Q}\beta)=2.$ By assumption, $\mathrm{zero}\left(  F\right)  $ is
non-empty and consequently, necessarily not finite. 

\item Suppose that $\dim_{\mathbb{Q}}(\mathbb{Q}+\mathbb{Q}\alpha
+\mathbb{Q}\beta)=3.$ By assumption, $\mathrm{zero}\left(  F\right)  $ is
non-empty and consequently necessarily dense in $\left[  0,1\right)  ^{2}.$
Since $F$ is continuous and vanishing on a set dense subset of $\left[
0,1\right)  ^{2}$ then must be everywhere vanishing. Next, appealing to the
fact that the Zak transform is unitary, it follows that $f$ must be the zero
vector in $L^{2}\left(  \mathbb{R}\right)  .$
\end{itemize}

\section*{Acknowledgment}

K.A.O. was partially supported by grants from the National Science Foundation
under grant numbers DMS 1814253 and DMS 2205771.
V.O. was partially supported by a grant from the National Science Foundation
under grant number DMS 2205852.
We extend our gratitude to Chris Heil and Shahaf Nitzan for identifying an
error in the previous version of our initial results. This led us to revise
and present our first result as articulated in Theorem~\ref{thm:amalgampertub}.

\bibliographystyle{amsplain}
\bibliography{1nconf_B}

\end{document}